\documentclass[12pt]{amsart}
\usepackage{ifthen,verbatim}
\usepackage{floatflt,graphicx,graphics}

\usepackage{tikz} 
\usepackage{mathrsfs}
\usepackage{amsmath,amssymb,amsthm}
\usepackage{mathtools}
\usepackage{pdfpages}
\usepackage{a4wide}
\usepackage{multirow}

\numberwithin{equation}{section}
\nonstopmode
\numberwithin{equation}{section}
\theoremstyle{plain}
\newtheorem{theorem}[equation]{Theorem}
\newtheorem{thm}[equation]{Theorem}
\newtheorem{conjecture}[equation]{Conjecture}
\newtheorem{lemma}[equation]{Lemma}
\newtheorem{corollary}[equation]{Corollary}

\newtheorem{prop}[equation]{Proposition}

\theoremstyle{definition}

\newtheorem{nonsec}[equation]{}

\theoremstyle{remark}

\newcommand{\R}{\mathbb{R}}

\newcommand{\C}{\mathbb{C}}

\newcommand{\B}{\mathbb{B}}

{\qed\bigskip}

\newcounter{alphabet}

\newcounter{minutes}
\divide\time by 60
\newcounter{hours}
\multiply\time by 60
\addtocounter{minutes}{-\time}

\begin{document}

\bibliographystyle{amsplain}
\title[Lipschitz constants and quadruple symmetrization]
{
Lipschitz constants and quadruple symmetrization by M\"obius
transformations
}

\author[O. Rainio]{Oona Rainio}
\author[M. Vuorinen]{Matti Vuorinen}

\keywords{Hyperbolic geometry, Lipschitz constants, M\"obius transformations, symmetrization}
\subjclass[2010]{Primary 51M10, Secondary 51M16}
\begin{abstract}
Due to the invariance properties of cross-ratio, M\"obius transformations are often used to map a set of points or other geometric object into a symmetric position to simplify a problem studied. However, when the points are mapped under a M\"obius transformation, the distortion of the Euclidean geometry is rarely considered. Here, we identify several cases where the distortion caused by this symmetrization can be measured in terms of the Lipschitz constant of the M\"obius transformation in the Euclidean or the chordal metric.
\end{abstract}
\maketitle

\section{Introduction}

Symmetry is an important property considered by many mathematicans working in various areas of research, including classical analysis, geometry, and solutions to extremal problems \cite{bae,ban,ber,du}. Problem-specific methods have been introduced with the purpose of transforming objects into symmetric ones \cite{bae}. Another use of symmetry is to reduce some complicated problem to its special case which exhibits symmetry and is therefore easier to solve than the original question.

For quadruples of points on the complex plane, a symmetrization can be performed by means of a M\"obius transformation. The use of M\"obius transformations is in fact a standard tool in classical function theory. However, while there are several metrics, such as the hyperbolic metric, that are invariant with respect to the M\"obius transformations, these mappings affect the Euclidean distances between the complex points considered. Our aim is to analyze "cost" of the symmetrization due to distortion under M\"obius transformations, measured in terms of Lipschitz constants. These Lipschitz constants can then be formulated with respect to Euclidean, hyperbolic, or chordal geometry.

We consider symmetrization of quadruples of points in several cases, under various constraints of the location of the points of the quadruple $\{a,b,c,d\}.$ In the most general case, we find a M\"obius transformation $h$ mapping the quadruple onto the quadruple $\{-1,-y,y,1\}$ symmetric with respect to the origin. For any such M\"obius
transformation $h$, it seems natural to call $h^{-1}(0)$ {\it a M\"obius center of the quadruple}. We give an explicit formula for such a M\"obius center.

\begin{theorem}\label{thm_mobcenter}
Let $h$ be the M\"obius transformation of $\overline{\R}^2$ that maps distinct points $a,b,c,d\in\R^2$ onto $-1,y,-y,1$, respectively. If $a,b,c,d$ are collinear and $|a-b|=|c-d|$, then $h^{-1}(0)=(a+d)/2$. Otherwise,
\begin{align*}
h^{-1}(0)=\frac{(b-c)(a+d)y+(b+c)(a-d)}{2((b-c)y+a-d)},    
\end{align*}
where $y$ is as in Lemma \ref{lem_symth}.
\end{theorem}

The structure of this paper is as follows. In Section 3, we study a certain chordal isometry and related Lipschitz constants of the chordal metric. In Section 4, we find a M\"obius transformation for the symmetrization of any four distinct points on the complex plane and prove Theorem \ref{thm_mobcenter}. In Section 5, we introduce three normalization methods for a pair of points in the unit disk and, in Section 6, we consider the symmetrization of quadruples of points on the unit circle.

\section{Preliminaries}

Let us first introduce the notations used in this paper. For a dimension $n\geq0$, $e_n$ is the $n$th unit vector of the real space $\R^n$, $\B^n$ is the unit ball, and $S^{n-1}$ is the unit sphere. For a point $x\in\R^n$ and a radius $r$, let $B^n(x,r)=\{y\in\R^n\,:\;|x-y|<r\}$ and $S^{n-1}(x,r)=\{y\in\R^n\,:\;|x-y|=r\}$. For any point $x\in\C$, denote its complex conjugate by $\overline{x}$. For two distinct points $x,y\in\R^n$, let $L(x,y)$ be the Euclidean line passing through them. Denote the extended real space by $\overline{\R}^n=\R^n\cup\{\infty\}$.

The intersection point of two non-parallel lines $L(a,b)$ and $L(c,d)$ is given by \cite[Ex. 4.3(1), p.\,57 \& p.\,373]{hkv}
\begin{equation}\label{eq:lis}
{\rm LIS}[a,b,c,d]=
   \frac{(\overline{a}b-a\overline{b})(c-d)
          -(\overline{c}d-c\overline{d})(a-b)}
        {(\overline{a}-\overline{b})(c-d)
          -(\overline{c}-\overline{d})(a-b)}.
\end{equation}

If the four points $a,b,c,d$ are on the unit circle, then we have $\overline{z}=1/z$ for $z\in\{a,b,c,d\}$ and the formula \eqref{eq:lis} can be simplified as in the following proposition.

\begin{prop}\cite[p.13]{p01}\label{prop_conjIntF}
If $a,b,c,d$ are four distinct complex points on the unit circle $S^1$ chosen so that the lines $L(a,b)$ and $L(c,d)$ are non-parallel, then the complex conjugate $\overline{f}$ of the intersection point $f$ of the lines $L[a,c]$ and $L[b,d]$ is
\begin{align*}
\overline{f}=\frac{a+c-b-d}{ac-bd}.
\end{align*}
\end{prop}

\begin{prop}\label{LISprop}
Let $a,b \in \mathbb{C}$ with $|a|\neq |b|, |a||b| \neq 1.$ Then
\begin{align*}
(1)& \quad \quad {\rm LIS}[a,b, -1/\overline{a}, -1/\overline{b}]= 
\frac{b(1+|a|^2)-a(1+|b|^2)}{|a|^2 -|b|^2},\\
(2)& \quad \quad {\rm LIS}[a,b, 1/\overline{a}, 1/\overline{b}]= 
\frac{a(1-|b|^2)-b( 1-|a|^2)}{|a|^2 -|b|^2},\\
(3)& \quad \quad {\rm LIS}[a, 1/\overline{b},b, 1/\overline{a}]= 
\frac{a(1-|b|^2)+b(1-|a|^2)}{1-|a|^2 |b|^2}\, ,\\
(4)& \quad \quad {\rm LIS}[a, -1/\overline{b},b, -1/\overline{a}]= 
\frac{a(1+|b|^2)+b(1+|a|^2)}{1-|a|^2 |b|^2}\,.    
\end{align*}
\end{prop}
\begin{proof}
The proof follows from \eqref{eq:lis}.
\end{proof}

\begin{prop}\label{UCLISprop}
Let $a,b \in \mathbb{B}^2$ be points non-collinear with $0$ and
$c={\rm LIS}[a,b,0, i(a-b)].$ Then $L[a,b]\cap S^1 =\{a_1,b_1\}$
where
$$  a_1=c-i \frac{c}{|c|}\sqrt{1-|c|^2}, \quad b_1=c+i \frac{c}{|c|}\sqrt{1-|c|^2},$$
and $a_1,a,b,b_1$ are ordered in such a way that $|a_1-a|<|a_1-b|.$
\end{prop}
\begin{proof}
The proof follows from \eqref{eq:lis}.
\end{proof}

\begin{nonsec}{\bf M\"obius transformations.}\label{mymob} \cite[Ex. 3.2, pp. 25-26 \& Def. 3.6, p. 27]{hkv}
For any $t\in\R$ and $u\in\R^n\backslash\{0\}$, the hyperplane perpendicular to the vector $u$ and at distance $t\slash|u|$ from the origin is
\begin{align*}
P(u,t)=\{x\in\R^n\text{ }|\text{ }x\cdot u=t\}\cup\{\infty\},   
\end{align*}
where $\cdot$ denotes the dot product. The reflection in this hyperplane is defined by the function $h:\overline{\R}^n\to\overline{\R}^n$,
\begin{align*}
h(x)=x-2(x\cdot u-t)\frac{u}{|u|^2},\quad h(\infty)=\infty. \end{align*}
Define then the inversion in $S^{n-1}(v,r)$ as $g:\overline{\R}^n\to\overline{\R}^n$,
\begin{align*}
g(x)=v+\frac{r^2(x-v)}{|x-v|^2},\quad g(v)=\infty,\quad g(\infty)=v.    
\end{align*}

A \emph{M\"obius transformation} is a function $f:\overline{\R}^n\to\overline{\R}^n$ that can be defined as a function composition $f=h_1\circ\cdots\circ h_m$ with an integer $m\geq1$ so that each $h_j$, $j=1,...,m$, is either a reflection in some hyperplane or an inversion in a sphere. The M\"obius transformation is called sense-preserving when $m$ here is even and sense-reversing when $m$ is odd. By \emph{Liouville's theorem}, any conformal mapping with a domain in $\R^n$, $n\geq3$, is a M\"obius transformation in $\overline{\R}^n$ restricted to this domain \cite[Rmk 3.44, p. 47]{hkv}.

In the two-dimensional case, the expression of a sense-preserving M\"obius transformation can be written as
\begin{align}\label{exp_mob}
z \mapsto \frac{az+b}{cz+d}\,, \quad a,b,c,d,z \in {\mathbb C},\quad ad-bc\neq 0.   
\end{align}
The expression of the sense-reversing M\"obius transformation can be obtained by replacing $z$ by $\overline{z}$ in \eqref{exp_mob}. The special M\"obius transformation
\begin{equation}\label{myTa}
T_a(z) = \frac{z-a}{1- \overline{a}z}\,, \quad a\in\B^2\setminus\{0\}\,,
\end{equation}
maps the unit disk $\mathbb{B}^2$ onto itself with $T_a(a) =0, T_a( \pm a/|a|)= \pm a/|a|\,.$ 
\end{nonsec}

Ahlfors \cite{a} gave the following factorization for $T_a$, $a\in\B^2\setminus\{0\}$, as a composition of an inversion followed by a reflection which is given here in the complex notation. Denote $a^*=a/|a|^2$. Let $r=\sqrt{|a|^{-2}-1}$ and $\sigma_a$ be the inversion in the circle $S^1(a^*,r)$ orthogonal to the unit circle $S^1$. Then $\sigma_{a}(a)=0$ and $\sigma_a(a^*)=\infty.$
Let $p_a$ be the reflection over the line through $0$ and 
perpendicular to $a.$ Then both $\sigma_a$ and $p_a$ are 
sense-reversing transformations and their composite transformation $T_a$ in \eqref{myTa} is a sense-preserving
M\"obius transformation with the decomposition
$$
T_a(z) = p_a(\sigma_a(z)), \quad \sigma_a(z)= \frac{a(1-(\overline{z}/ \overline{a}))}{1-  a \overline{z}},\quad p_a(z) =- \frac{a}{\overline{a}}\overline{z}.
$$
Clearly, $p_a$ is a Euclidean isometry and hence the information about distance distortion under $T_a$ is due to 
the mapping $\sigma_a.$
Moreover, for  $a,b\in {\mathbb B}^2 \setminus \{0\}\,, |a|\neq |b|,$  a M\"obius transformation $T_{a,b}: \B^2\to \B^2$ with
$ T_{a,b}(a) = b$ is given by
\begin{equation}\label{myTab}
T_{a,b}(z)= \frac{b}{ \overline{b}} \frac{\overline{c} z-1}{z-c},\quad c= \frac{a-b + a b ( \overline{a}- \overline{b})}{|a|^2-|b|^2}.
\end{equation}
The above simple formula for the point $c$  follows from  Proposition \ref{LISprop}(2)
 for the intersection of two lines 
$
c={\rm LIS}[a,b,a^*,b^*].
$
For $a,b\in\B^2\setminus\{0\}$, the M\"obius transformation $T_{a,b}$ is a reflection over the line $L(0,b)$ composed with the inversion in $S^1(c,r)$ where $r=\sqrt{|c|^2-1}$ so that $S^1(c,r)$ is orthogonal to the unit circle. See Figure \ref{fig1}.

\begin{thm} \label{LIP} Let $T_a$ and $T_{a,b}$ be the above M\"obius transformations \eqref{myTab}.
Then
\begin{equation}\label{myLip2}
\mbox{Lip}(T_a)= \frac{1+ |a|}{1-|a|}\,.
\end{equation}
and
\begin{equation}\label{LipTab}
\mbox{Lip}(T_{a,b})= \frac{|c|+1}{|c|-1}\,, \quad  
c= \frac{a-b + a b ( \overline{a}- \overline{b})}{|a|^2-|b|^2},
\end{equation}
where the Lipschitz constant are defined for the Euclidean metric.
\end{thm}
\begin{proof} The formula \eqref{myLip2} was proved in   \cite[p.\,43]{b}. The second part follows
from \eqref{myTab}.
\end{proof}

\begin{figure}[ht]
    \centering
    \begin{tikzpicture}[scale=2.1]
    \clip (-1.15,-2.4) rectangle (2.3,1.1);
    \draw (0,0) circle (0.03cm);
    \draw (0,-0.7) circle (0.03cm);
    \draw (0.5,0) circle (0.03cm);
    \draw (0,-1.428) circle (0.03cm);
    \draw (2,0) circle (0.03cm);
    \draw (-1.062,-2.187) circle (0.03cm);
    \draw (3.625,4.375) -- (-5.750,-8.750);
    \draw (8.125,4.375) -- (-10.250,-8.750);
    \draw (0.977,0.208) arc (102.0:177.1:1.302);
    \draw[dashed] (-1.062,-2.187) -- (0.977,0.208);
    \draw (0,0) circle (1cm);
    \draw (-1.062,-2.187) circle (2.216cm);
    \node[scale=1.3] at (0,0.15) {$0$};
    \node[scale=1.3] at (-0.07,-0.57) {$a$};
    \node[scale=1.3] at (0.45,0.15) {$b$};
    \node[scale=1.3] at (0,-1.288) {$a^*$};
    \node[scale=1.3] at (2,0.15) {$b^*$};
    \node[scale=1.3] at (-1.09,-2.037) {$c$};
    \end{tikzpicture}
    \caption{The points $a,b,a^*,b^*$ when $a=-0.7i$ and $b=0.5$, the intersection point $c={\rm LIS}[a,b,a^*,b^*]$ of the lines $L(a,b)$ and $L(a^*,b^*)$, the unit circle, the circle $S^1(c,r)$ for $r=\sqrt{|c|^2-1}$, and the hyperbolic line $J^*[a,b]$. The end points of $J^*[a,b]$ are collinear with $c$, as denoted with dashed line. The M\"obius transformation $T_{a,b}$ is a reflection over the line $L(0,b)$ composed with the inversion in the circle $S^1(c,r)$.}
    \label{fig1}
\end{figure}
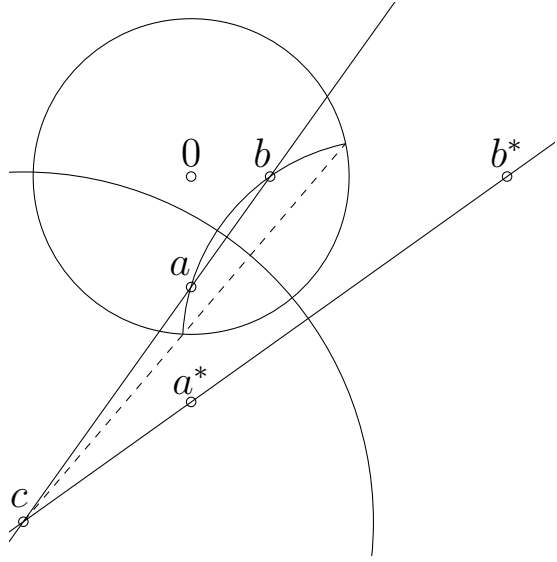

\begin{nonsec}{\bf Hyperbolic geometry.} Denote hyperbolic sine, cosine and tangent by sh, ch, and th, respectively. The hyperbolic metric $\rho_{\B^n}$ in the unit ball $\B^n$ is defined by
\begin{equation}\label{myrho}
{\rm sh}\frac{\rho_{\B^n}(x,y)}{2}= \frac{|x-y|}{\sqrt{(1-|x|^2)(1-|y|^2)}}\,,
\end{equation}
and, if $n=2$, this is equivalent to
\begin{align}
{\rm th}\frac{\rho_{\B^2}(x,y)}{2}=\frac{|x-y|}{|1-x\overline{y}|}=\frac{|x-y|}{A[x,y]}, 
\end{align}
where $A[x,y]$ is the Ahlfors bracket \cite[7.37]{avv} 
\begin{equation}\label{myahl}
A[x,y]=|1-x\overline{y}|=\sqrt{|x-y|^2+(1-|x|^2)(1-|y|^2)}.
\end{equation}
We denote the hyperbolic line through $x,y\in\B^2$ by $J^*[x,y]$. If $x,y$ are collinear with the origin, the line $J^*[x,y]$ is a diameter of the unit disk and, otherwise, an arc of the circle that passes through $x,y$ and is orthogonal to the unit circle \cite{b}. For $x,y\in\B^2\setminus\{0\}$, the end points of the hyperbolic line $J^*[x,y]$ can be expressed as $ep(x,y)$ and $ep(y,x)$ where 
\begin{equation}\label{epdef}
ep(x,y)=T_{-y}(T_y(x)/|T_y(x)|).
\end{equation}
Define then the \emph{absolute ratio} of any four distinct points $a,b,c,d\in\R^n$ as \cite{b}
\begin{align}
|a,b,c,d|=\frac{|a-c||b-d|}{|a-b||c-d|}.
\end{align} 
The hyperbolic metric fulfills
\begin{equation}\label{epdefrho}
\rho_{\B^2}(x,y)=\log|ep(x,y),x,y,ep(y,x)|,
\end{equation}
which shows that the hyperbolic metric is invariant under M\"obius transformations of the unit disk onto itself.
\end{nonsec}

\begin{thm}\label{myhmidp}\cite[Thm 1.4, p. 126]{wvz}
For given $x,y\in\B^2$, the hyperbolic midpoint $z\in\B^2$ with
$\rho_{\B^2}(x,z)=\rho_{\B^2}(y,z)=\rho_{\B^2}(x,y)/2$ is given by
\begin{equation}\label{myzformua}
z=\frac{y(1-|x|^2) + x(1-|y|^2)}{1-|x|^2|y|^2 + A[x,y] \sqrt{(1-|x|^2)(1-|y|^2)}}.
\end{equation}
\end{thm}

\begin{nonsec}{\bf Chordal geometry.}
Define the {\it stereographic
projection } \ $\ \pi :\overline{\R}^2\to S^2(\frac{1}{2} e_3,\frac{1}{2} )$, $e_3=(0,0,1)$,
as \cite[(3.4), p. 28]{hkv}
\begin{equation}\label{1.13.}
  \pi (x)= e_3+\frac{x-e_3}{ |x-e_3|^2} \,,\;x\in {\R}^2\;;\;\,
  \pi (\infty)=e_3\;.
\end{equation}
The mapping $\pi$ is now the restriction to $\overline{\R}^2$ of the inversion in $S^2(e_3,1)$ and, since $f^{-1}=f$  for every inversion  $f$, $\pi$ maps the ``Riemann sphere'' $S^2(\frac{1}{2} e_3,\frac{1}{2} )$
onto  $\overline{\R}^2$. By using the stereographic projection $\pi$, we can define the {\it spherical (chordal) metric} $q$ in $\overline{\R}^2$ as \cite[(3.5), p. 29]{hkv}
\begin{equation}\label{1.14.}
   q(x,y)=|\pi(x)-\pi(y)|\;;\;\, x,\,y\in  \overline{\R}^2\;,
\end{equation}
which can be extended to the $n$-dimensional case so that \cite[(3.6), p. 29]{hkv}
\begin{equation}\label{1.15.}
  \begin{cases}
     {\displaystyle
			q(x,y)=\frac{|x-y|}{\sqrt{1+|x|^2}\;\sqrt{1+|y|^2}}}\;;\;\, x,\,y\in\R^n\;,&\\
	{\displaystyle 
	q(x,\infty)=\frac{1}{\sqrt{1+|x|^2}}}\;;\;\,x\in\R^n.&
  \end{cases}
\end{equation}
\end{nonsec}
For $x\in\R^n$ and $r>0$, denote the chordal ball by $B_q(x,r)=\{y\in\R^n\,:\;q(x,y)<r\}$. 

\section{Lipschitz constants in chordal metric}

In this section, we introduce an isometry $t_m$ of the chordal metric. Note that, while the mapping $T_a$ is an isometry for the hyperbolic metric, it is not a chordal isometry. However, our computer tests suggest that it has the following Lipschitz constant. 

\begin{conjecture}
For $a\in\B^n$, the Lipschitz constant of the mapping $T_a$ in the chordal metric is
\begin{align*}
Lip(T_a|\overline{\R}^n)\equiv
\sup\left\{\frac{q(T_a(x),T_a(y))}{q(x,y)}\,:\,x,y\in\overline{\R}^n\right\}
=\frac{1+|a|}{1-|a|}.     
\end{align*}
\end{conjecture}

It is well-known that $(1+|a|)/(1-|a|)$ is also the Lipschitz constant $Lip(T_a|\B^n)$ for the Euclidean metric, as stated in Theorem \ref{LIP}.

Let us then define the chordal isometry. Recall the notation $a^*$ for $a\in\R^2\setminus\{0\}$ and $p_a$ from Subsection \ref{mymob}. Let $s_a$ be the inversion in the circle $S^1(-a^*,\sqrt{1+|a|^{-2}})$. Now, we can define a chordal isometry as \cite[(3.23), p. 40]{hkv}
\begin{align}\label{chrismtry}
t_a=p_a\circ s_a.    
\end{align}
Trivially, $t_a$ is a M\"obius transformation and it satisfies $t_a(a)=0$. In fact, we have the following formula \cite[(B.9), p. 459]{hkv}
\begin{align*}
t_a(z)=\frac{z-a}{1+\overline{a}z}.    
\end{align*}
Furthermore, we have \cite[(B.10), p. 459]{hkv}
\begin{align*}
t_a(z)=-t_{-a}(-z).    
\end{align*}

\begin{theorem}\label{thm_cmp}\cite[Thm 5.12]{fcp}
For $ a,b\in\R^2$, the chordal midpoint $ m $ is given by
$$
  m=\frac{a(1+|b|^2)+b(1+|a|^2)}
         {|1+a\overline{b}|\sqrt{(1+|a|^2)(1+|b|^2)}-|ab|^2+1}.
$$
\end{theorem}

We see from the above formula that, if $a,b\in\mathbb{B}^2$, their chordal midpoint $m\in\B^2$.  

\begin{corollary}
For $a,b\in\R^2$, let $m$ be their chordal midpoint as in Theorem \ref{thm_cmp}. Then the chordal isometry $t_m:\overline{\R}^2\to\overline{\R}^2$ has the formula
\begin{align*}
t_m(z)=\frac{\big(|1+a\overline{b}|\sqrt{(1+|a|^2)(1+|b|^2)}+(1-|ab|^2)\big)z
               -ab(\overline{a}+\overline{b})-(a+b)}
             {\big(\overline{ab}(a+b)+(\overline{a}+\overline{b})\big)z+
              |1+a\overline{b}|\sqrt{(1+|a|^2)(1+|b|^2)}+1-|ab|^2}.    
\end{align*}
Furthermore, we have
\begin{align*}
t_m(a) &= \frac{a|1+a\overline{b}|\sqrt{1+|b|^2}
                    -b(1+a\overline{b})\sqrt{1+|a|^2}}
                 {|1+a\overline{b}|\sqrt{1+|b|^2}
                     +(1+a\overline{b})\sqrt{1+|a|^2}},\\ 
t_m(b) &= \frac{b|1+\overline{a}b|\sqrt{1+|a|^2}
                   -a(1+\overline{a}b)\sqrt{1+|b|^2}}
                 {|1+\overline{a}b|\sqrt{1+|a|^2}
                    +(1+\overline{a}b)\sqrt{1+|b|^2}}.    
\end{align*}
\end{corollary}
\begin{proof}
The result follows by inputting $m$ as in Theorem \ref{thm_cmp} into $t_m(z)=(z-m)/(1+\overline{m}z)$. 
\end{proof}

\section{M\"obius symmetrization}

In this section, we prove a few results related to using a M\"obius transformation for the symmetrization of quadruples of points on the complex plane.

\begin{lemma}\cite[Lemma 3.17, p. 34]{hkv}
Let $f$ be a M\"obius transformation of $\overline{\R}^n$ that maps a quadruple 0, $e_1$, $x$, and $\infty$ of distinct points onto the quadruple $-e_1$, $y$, $-y$, and $e_1$, respectively, such that $|y|\leq1$. Then
\begin{align*}
|y|=\frac{|x-e_1|}{1+|x|+t}\quad\text{and}\quad|y+e_1|^2=\frac{|y-e_1|^2}{|x|}=\frac{4}{1+|x|+t},    
\end{align*}
where $t=\sqrt{(1+|x|)^2-|x-e_1|^2}$.
\end{lemma}

\begin{lemma}\label{lem_symth}
For four distinct points $a,b,c,d\in\R^2$, let $h$ be the M\"obius transformation of $\overline{\R}^2$ that maps $a,b,c,d$ onto $-1,y,-y,1$, respectively, for some point $y\in\R^2$. If $a,b,c,d$ are collinear and $|a-b|=|c-d|$, then
\begin{align*}
h(z)=-\frac{2}{a-d}\,z+\frac{a+d}{a-d}.    
\end{align*}
Otherwise,
\begin{align*}
h(z)&=\frac{pz+q}{z+s}\quad\text{with}\\
p&=-\frac{(b-c)y+a-d}{a-b-c+d},\quad
q=\frac{(b-c)(a+d)y+(b+c)(a-d)}{2(a-b-c+d)},\\
s&=-\frac{(b+c)(a-d)y+(b-c)(a+d)}{2((a-d)y+(b-c))},
\end{align*}
where 
\begin{align*}
y&=\frac{-k_1\pm\sqrt{k_1^2-k_0^2}}{k_0}\quad\text{with}\\
k_0&=(b-c)(a-d),\quad
k_1=(a-c)(b-d)+(a-b)(c-d).
\end{align*}
The two possible solutions for $y$ above are reciprocal numbers so we can choose $y$ so that $|y|<1$ in this case.
\end{lemma}
\begin{proof}
First, if $a,b,c,d$ are collinear and satisfy $|a-b|=|c-d|$,
then $h$ is a linear map and can be written $h(z)=uz+v$. Since
\begin{align*}
h(a)=ua+v=-1,\quad h(d)=ud+v=1,    
\end{align*}
we have
\begin{align*}
v=-\frac{(a+d)u}{2},\quad u=-\frac{2}{a-d}.    
\end{align*}
Thus, the first part of the lemma follows.

Next, let us consider the other case. We may assume that $h(z)=\dfrac{pz+q}{z+s}$ since $h$ is clearly not a linear map. From the conditions $h(a)=-1,\ h(b)=y,\ h(c)=-y,\ h(d)=1$, we obtain
\begin{align} 
\label{A} &ap+q+s+a=0, \\
\label{B} &-(s+b)y+bp+q=0, \\
\label{C} &(s+c)y+cp+q=0, \\
\label{D} &dp+q-s-d=0.
\end{align}
Calculating \eqref{A}+\eqref{D}, \eqref{A}-\eqref{D}, \eqref{B}+\eqref{C}, and \eqref{B}-\eqref{C} results in the following equations, respectively: 
\begin{align}
\label{A+D} &  (a+d)p+2q+a-d=0,\\
\label{A-D} & (a-d)p+2s+a+d=0,\\
\label{B+C} & (-b+c)y+(b+c)p+2q=0,\\
\label{B-C} & (-2s-b-c)y+(b-c)p=0.
\end{align}
By eliminating $q$ from \eqref{A+D} and \eqref{B+C}, we have
\begin{align*}
Y_1\equiv(b-c)y+(a-b-c+d)p+a-d=0\quad\Leftrightarrow\quad
p=-\frac{(b-c)y+a-d}{a-b-c+d}.
\end{align*}
Similarly, by eliminating $p$ from \eqref{A+D} and \eqref{B+C}, we have
\begin{align*}
&(b-c)(a+d)y-2(a-b-c+d)q+(b+c)(a-d)=0,\\
\Leftrightarrow\quad &q=\frac{(b-c)(a+d)y+(b+c)(a-d)}{2(a-b-c+d)}.
\end{align*}
By eliminating $s$ from \eqref{A-D} and \eqref{B-C}, we have
\begin{align*}
Y_2\equiv((a-d)p+a-b-c+d)y+(b-c)p=0.
\end{align*}
Elimination of $p$ from \eqref{A-D} and \eqref{B-C} gives us
\begin{align*}
&(2(a-d)s+(b+c)(a-d))y+2(b-c)s+(b-c)(a+d)=0\\
\Leftrightarrow\quad &s=-\frac{(b+c)(a-d)y+(b-c)(a+d)}{2((a-d)y+(b-c))}.
\end{align*}
Finally, by eliminating $p$ from $ Y_1=0 $ and $Y_2=0$, we have
$$
(b-c)(a-d)y^2+2((a-c)(b-d)+(a-b)(c-d))y+(b-c)(a-d)=0,
$$
which is equivalent to $k_0y^2+2k_1+k_0=0$ where $k_0$ and $k_1$ are as in the lemma. The latter part of our lemma follows from this.
\end{proof}

The M\"obius transformation of the former lemma is a very useful mapping as it can be used for symmetrization. Namely, with this transformation, we can map the arbitrary quadruple of points $a,b,c,d$ onto the points $-1,y,-y,1$ that are clearly symmetric with respect to the origin. However, as the origin is not preserved under the mapping $h$, this raises the question about the location of the M\"obius center $h^{-1}(0)$ in terms of $a,b,c,d$. We can now prove Theorem \ref{thm_mobcenter} from Introduction. 

\begin{nonsec}{\bf Proof of Theorem \ref{thm_mobcenter}}
\begin{proof}
The mapping $h$ is as in Lemma \ref{lem_symth}. If $a,b,c,d$ are collinear and $|a-b|=|c-d|$, then it is a linear map and $h^{-1}(0)$ is the midpoint of $a$ and $d$. Consider then the other case. Denote $w=h^{-1}(0)$. Now, for $p,q,s$ as in Lemma \ref{lem_symth},
\begin{align*}
\frac{pw+q}{w+s}=0\quad\Leftrightarrow\quad
w=-\frac{q}{p}=\frac{(b-c)(a+d)y+(b+c)(a-d)}{2((b-c)y+a-d)}.
\end{align*}
\end{proof}
\end{nonsec}

\section{Three normalizations of point pairs}\label{go}
Mathematical proofs involving distances are often simplified if the configurations studied
exhibit some kind of symmetry or special structure. We study here, for a given pair of
distinct points $a,b $ in $\B^2\setminus\{0\}$ , three natural normalizations:
\begin{enumerate}
\item[(1)] $a=-b,$
\item[(2)] $a,b,0$ are collinear,
\item[(3)] $|a|=|b|\,.$
\end{enumerate}
If the points are in general position, then they can be mapped onto each of the
above special positions by means of a M\"obius transformation. The cost is that
the geometry is changed in each case: we measure the cost in terms of the
Lipschitz constant of the M\"obius transformation.

A simple way to achieve the normalization (2) above is to relabel the points $a$ and $b$ such that 
$|a|\le |b|$ and to use the mapping $T_a.$ By Theorem \ref{LIP} the cost of normalization can be
immediately obtained. In our discussion below, the plan is to carry out the normalization with a smaller
cost, if possible.

\begin{thm} \label{3norm} Consider  a given pair of
distinct points $a,b$ in $\B^2\setminus\{0\}$. Let $k$ be the point on the hyperbolic line $J^*[a,b]$ that is closest to the origin. In other words, $k=0$ if $a,b,0$ are collinear and, if $a,b,0$ are not collinear, then
$$
k=\frac{o_{ab}}{t}(t-\sqrt{t^2-1}),\quad t=|o_{ab}|,
$$
where let $o_{ab}$ be the center of the circle through $a$ and $b$, orthogonal to the unit circle. Let $h_{ab}$ be the hyperbolic midpoint of $a$ and $b$. Then there are  M\"obius transformations
$m_j:\B^2\to\B^2, j=1,2,3$ with the following properties: $m_1(a)= -m_1(b)$ with
\begin{equation} \label{norm1}
Lip(m_1)\le Lip(T_{h_{ab}})=\frac{1+|h_{ab}|}{1-|h_{ab}|},
\end{equation} 
$m_2(a), m_2(b),0$ are collinear and
\begin{equation} \label{norm2}
Lip(m_2)\le  Lip(T_k)=\frac{1+|k|}{1-|k|}, \quad
\end{equation} 
and 
$|m_3(a)|=| m_3(b)|$ and 
\begin{equation}
Lip(m_3)\le Lip(T_{h_{ab},k}), 
\end{equation} 
where $T_{h_{ab},k}$ is the mapping defined in \eqref{myTab}.
\end{thm}
\begin{proof}
We can choose for instance the mappings $m_1=T_{h_{ab}}$, $m_2=T_k$, and $m_3=T_{h_{ab},k}$ in which case the following inequalities of the Lipschitz constants $Lip(m_j)$ presented in the theorem hold with equality.
\end{proof}

For all $a,b\in\B^2\setminus\{0\}$, $|k|\leq|h_{ab}|$ where $h_{ab}$ and $k$ are as in Theorem \ref{3norm}. Namely, the hyperbolic point $h_{ab}$ trivially is on the hyperbolic line $J^*[a,b]$, but, unlike $k$, $h_{ab}$ is not necessarily the closest point to the origin on this line. It follows from the inequality $|k|\leq|h_{ab}|$ that the Lipschitz constants $Lip(T_{h_{ab}})$ and $Lip(T_k)$ in Theorem \ref{3norm} satisfy
\begin{align}
Lip(T_k)\leq Lip(T_{h_{ab}}) 
\end{align}
for all $a,b\in\B^2$. However, a similar inequality cannot be presented for the Lipschitz constant $Lip(T_{h_{ab},k})$. Namely, for $a=0.3$ and $b=0.1i$,
\begin{align*}
Lip(T_{h_{ab}})\approx1.379,\quad
Lip(T_k)\approx1.207,\quad
Lip(T_{h_{ab},k})\approx1.301,    
\end{align*}
but, for $a=0.3$ and $b=0.1i$,
\begin{align*}
Lip(T_{h_{ab}})\approx1.508,\quad
Lip(T_k)\approx1.508,\quad
Lip(T_{h_{ab},k})\approx2.029.    
\end{align*}
In other words, $Lip(T_{h_{ab},k})$ is sometimes less than the other two Lipschitz constants and sometimes greater than them.

\begin{figure}
    \centering
    \begin{tikzpicture}[scale=2.1]
    \draw (0,0) circle (0.03cm);
    \draw (0,-0.9) circle (0.03cm);
    \draw (0.5,-0.3) circle (0.03cm);
    \draw (0.296,-0.405) circle (0.03cm);
    \draw (0.736,-1.005) circle (0.03cm);
    \draw (0,0) -- (0.736,-1.005);
    \draw (0.955,-0.294) arc (72.8:179.5:0.744);
    \draw (0,0) circle (1cm);
    \node[scale=1.3] at (0,0.15) {$0$};
    \node[scale=1.3] at (0.13,-0.85) {$a$};
    \node[scale=1.3] at (0.47,-0.15) {$b$};
    \node[scale=1.3] at (0.1,-0.35) {$k$};
    \node[scale=1.3] at (0.97,-1) {$o_{ab}$};
    \draw (3,0) circle (0.03cm);
    \draw (2.325,-0.494) circle (0.03cm);
    \draw (3.249,0.183) circle (0.03cm);
    \draw (2.194,-0.590) -- (3.806,0.590);
    \draw (3,0) circle (1cm);
    \node[scale=1.3] at (2.97,0.17) {$0$};
    \node[scale=1.3] at (2.7,-0.5) {$T_k(a)$};
    \node[scale=1.3] at (3.6,0.15) {$T_k(b)$};
    \end{tikzpicture}
    \caption{The M\"obius transformation $T_k$ maps the nearest point $k$ on the hyperbolic line $J^*[a,b]$
    through $a=-0.9i$ and $b=0.5-0.3i$ to the origin so that the images of $a$ and $b$ are collinear with the origin.}
    \label{fig2}
\end{figure}
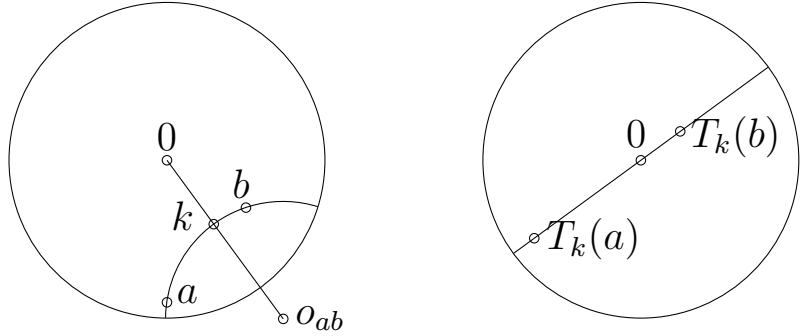

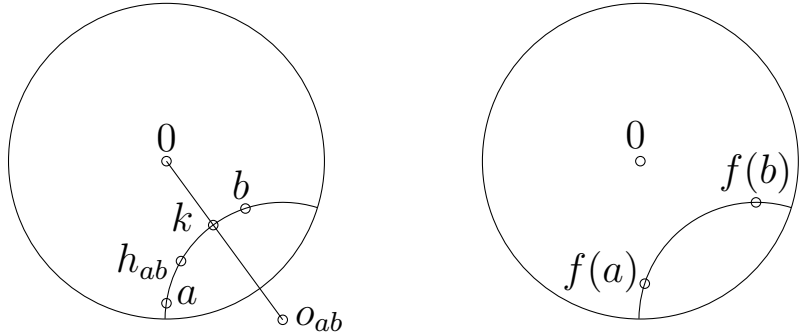
\begin{figure}
    \centering
    \centering
    \begin{tikzpicture}[scale=2.1]
    \draw (0,0) circle (0.03cm);
    \draw (0,-0.9) circle (0.03cm);
    \draw (0.5,-0.3) circle (0.03cm);
    \draw (0.296,-0.405) circle (0.03cm);
    \draw (0.092,-0.633) circle (0.03cm);
    \draw (0.736,-1.005) circle (0.03cm);
    \draw (0,0) -- (0.736,-1.005);
    \draw (0.955,-0.294) arc (72.8:179.5:0.744);
    \draw (0,0) circle (1cm);
    \node[scale=1.3] at (0,0.15) {$0$};
    \node[scale=1.3] at (0.13,-0.85) {$a$};
    \node[scale=1.3] at (0.47,-0.15) {$b$};
    \node[scale=1.3] at (0.1,-0.35) {$k$};
     \node[scale=1.3] at (-0.15,-0.633) {$h_{ab}$};
    \node[scale=1.3] at (0.97,-1) {$o_{ab}$};
    \draw (3,0) circle (0.03cm);
    \draw (3.028,-0.776) circle (0.03cm);
    \draw (3.732,-0.261) circle (0.03cm);
    \draw (3.955,-0.294) arc (72.8:179.5:0.744);
    \draw (3,0) circle (1cm);
    \node[scale=1.3] at (2.97,0.17) {$0$};
    \node[scale=1.3] at (2.75,-0.7) {$f(a)$};
    \node[scale=1.3] at (3.73,-0.07) {$f(b)$};
    \end{tikzpicture}
    \caption{The M\"obius transformation $f=T_{h_{ab},k}$ maps the hyperbolic midpoint $h_{ab}$ of $a$ and $b$ to the nearest point $k$ on the hyperbolic line $J^*[a,b]$
    through $a=-0.9i$ and $b=0.5-0.3i$ so that the images of $a$ and $b$ are at the same distance from the origin.}
    \label{fig3}
\end{figure}

\section{Quadruples of points on the unit circle}

Consider four distinct points $a,b,c,d$ on the unit circle in the positive order. We can symmetrize the quadruple in two different
ways using a M\"obius transformation of the unit disk:
\begin{enumerate}
\item[(1)] After the symmetrization, the new quadruple is symmetric with respect to a diameter of the unit disk.
\item[(2)] After the symmetrization, the points of the new quadruple form the vertices of a rectangle.
\end{enumerate}
We will next show how to construct the corresponding M\"obius transformations and estimate the cost of symmetrization, the
Lipschitz constant of the transformation.

\begin{thm}\label{myoarcisec} \cite{wvz}
(1) Let
$$w_1={\rm LIS}[a,b,c,d]\,,w_2={\rm LIS}[a,c,b,d]\,,w_3={\rm LIS}[a,d,b,c]\,.$$
Then the point $w_2$ is the orthocenter of the triangle with vertices at the points $0$, $w_1$, $w_3.$

(2)
The point of intersection of the hyperbolic lines $J^*[a,c]$ and $J^*[b,d]$ is given by
\begin{equation*}\label{myoaisec}
w_4= \frac{(ac-bd)\pm\sqrt{(a-b)(b-c)(c-d)(d-a)}}{a-b+c-d}\,,
\end{equation*}
where the sign "$+$" or "$-$" in front of the square root is chosen such that $|w_4|<1\,.$
 \end{thm}

\begin{thm}\label{quadruples}
Let  $a,b,c,d$ be four distinct points on the unit circle in the positive order. Let $w_4$ be the point in Theorem \ref{myoarcisec}(2) and $w_5$ be the midpoint of the orthogonal arc with center at $w_1={\rm LIS}[a,b,c,d].$ There exists a M\"obius transformation $m_4:\B^2\to\B^2$ such that $m_4(a) , m_4(b), m_4(c),m_4(d)$
are symmetric with respect to a line through $0$ and
\begin{equation}\label{quad1}
Lip(m_4)= Lip(T_{w_4,w_5}).
\end{equation}
Moreover, if $m_5= T_{w_4},$ then $m_5(a) =- m_5(c), m_5(b)=-m_5(d).$
\end{thm}

For a given point $w_1\in\R^2\setminus\overline{\B}^2$, a circle $S^1(r,w_1)$, $r>0$, is orthogonal to the unit circle if and only if its radius $r$ is $\sqrt{|w_1|^2-1}$. Consequently, the point $w_5$ in Theorem \ref{quadruples} has the formula
\begin{align}
w_5=\frac{w_1}{|w_1|}(|w_1|-r),\quad\sqrt{|w_1|^2-1}.     
\end{align}
The intersection point $w_1={\rm LIS}[a,b,c,d]$ can be found with Proposition \ref{prop_conjIntF}.

\begin{figure}
    \centering
    \begin{tikzpicture}[scale=3]
    \draw (0,0) circle (0.03cm);
    \draw (1,0) circle (0.03cm);
    \draw (0.955,0.295) circle (0.03cm);
    \draw (0.070,0.997) circle (0.03cm);
    \draw (-0.504,0.863) circle (0.03cm);
    \draw (0.822,1.172) circle (0.03cm);
    \draw (0.354,0.260) circle (0.03cm);
    \draw (0.233,0.332) circle (0.03cm);
    \draw (0,0) -- (0.822,1.172);
    \draw (1,0) -- (0.822,1.172);
    \draw (-0.504,0.863) -- (0.822,1.172);
    \draw (0,0) circle (1cm);
    \draw (-0.185,0.982) arc (190.6:279.2:1.025);
    \draw (0.070,0.997) arc (175.9:270:0.931);
    \draw (-0.504,0.863) arc (210.3:287.1:1.260);
    \node[scale=1.3] at (-0.07,0.1) {$0$};
    \node[scale=1.3] at (1.1,0) {$a$};
    \node[scale=1.3] at (1.05,0.295) {$b$};
    \node[scale=1.3] at (0.070,1.1) {$c$};
    \node[scale=1.3] at (-0.504,1) {$d$};
    \node[scale=1.3] at (0.822,1.3) {$w_1$};
    \node[scale=1.3] at (0.354,0.13) {$w_4$};
    \node[scale=1.3] at (0.12,0.23) {$w_5$};
    \end{tikzpicture}
    \caption{The points $w_4$ and $w_5$ defining the M\"obius transformation $T_{w_4,w_5}$ of Theorem \ref{quadruples} when $a=1$, $b=e^{0.3i}$, $c=e^{1.5i}$, and $d=e^{2.1i}$.}
    \label{fig4}
\end{figure}
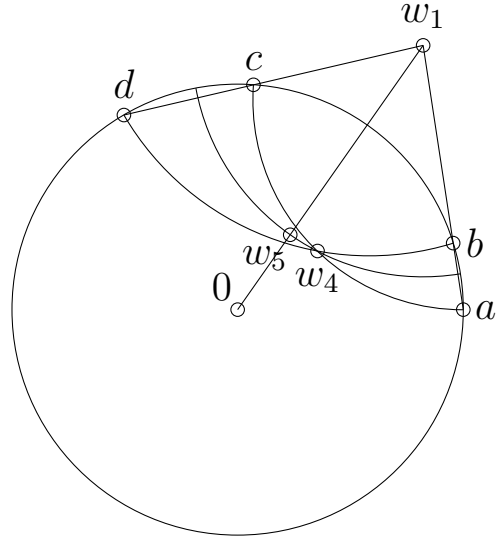

If $a=e^{1i}$, $b=e^{3i}$, $c=e^{4i}$, and $d=e^{5i}$, we have
\begin{align*}
Lip(T_{w_4,w_5})\approx1.496,\quad Lip(T_{w_4})\approx1.833,   
\end{align*}
and, if $a=e^{6i}$, $b=e^{6.1i}$, $c=e^{6.13i}$, and $d=e^{6.15i}$, we have
\begin{align*}
Lip(T_{w_4,w_5})\approx449.5,\quad Lip(T_{w_4})\approx60.00.   
\end{align*}
Consequently, neither of the mappings in Theorem \ref{quadruples} is always better than the other in terms of its Lipschitz constant. Computational tests suggest that $T_{w_4}$ produces a smaller Lipschitz constant if the points $a,b,c,d$ are close to each other, while $T_{w_4,w_5}$ is better when they are spread more evenly around the whole unit circle. This observation could be also concluded from the geometric meaning of the points $w_4$ and $w_5$.\\
\\

\noindent \textbf{Availability of data and materials.} Not applicable.\\
\textbf{Competing interests.} There are no competing interest.\\
\textbf{Funding.} O.R.'s research was funded by Magnus Ehrnrooth Foundation.\\
\textbf{Authors' contributions.} O.R. wrote the manuscript text, proved some of the results, and prepared all the figures, and M.V. proved rest of the results and supervised the project.\\
\textbf{Acknowledgements.} The authors are grateful to Professor Masayo Fujimura for her help in the proof of Lemma \ref{lem_symth}.\\
\textbf{Authors' information.}\newline
Oona Rainio$^1$, email: \texttt{ormrai@utu.fi}, ORCID: 0000-0002-7775-7656,\newline 
Matti Vuorinen$^1$, email: \texttt{vuorinen@utu.fi}, ORCID: 0000-0002-1734-8228\newline
1: University of Turku, FI-20014 Turku, Finland\\

\bibliographystyle{siamplain}

\end{document}